\documentclass[11pt,reqno]{amsart}
\usepackage{fullpage,amssymb,amsmath,mathrsfs,color}
\usepackage[final]{hyperref}

\newtheorem{theorem}{Theorem}[section]
\newtheorem{proposition}[theorem]{Proposition}

\newtheorem{lemma}[theorem]{Lemma}
\newtheorem{cor}[theorem]{Corollary}
\newtheorem*{question}{Question}

\theoremstyle{definition}
\newtheorem{definition}[theorem]{Definition}
\newtheorem{convention}[theorem]{Convention}
\newtheorem{remark}[theorem]{Remark}
\newtheorem{setting}[theorem]{Setting}
\newtheorem{notation}[theorem]{Notation}

\def\P{\mathbf P}
\newcommand\setsep{;\ }
\def\h{\hat{\text{ }}}
\def\rest{\hskip-2.5pt\restriction}
\def\sspan{\operatorname{span}}
\def\rng{\operatorname{Rng}}
\def\dom{\operatorname{Dom}}
\def\d{\operatorname{d}}
\def\diam{\operatorname{diam}}
\def\ov{\overline}
\def\S{\omega^{<\omega}}
\def\en{\mathbb N}
\def\er{\mathbb R}
\def\qe{\mathbb Q}
\def\zet{\mathbb Z}
\def\B{\mathcal B}
\def\F{\mathscr{F}}

\begin{document}

\author{Marek C\'uth}
\author{Martin Rmoutil}
\author{Miroslav Zelen\'y}
\title{On Separable Determination of $\sigma$-$\P$-Porous Sets \\ in Banach Spaces}

\thanks{M.\,C\'uth was supported by the Grant No. 282511/B-MAT/MFF of Grant Agency of Charles University in Prague.
M.\,Rmoutil was supported by the Grant No. 710812/B-MAT/MFF of Grant Agency of Charles University in Prague.
M.\,Zelen\'y was supported by the grant P201/12/0436.}

\email{marek.cuth@gmail.com, caj@rmail.cz, zeleny@karlin.mff.cuni.cz}
\address{Charles University, Faculty of Mathematics and Physics, Department of Mathematical\linebreak Analysis, Sokolovsk\'a 83, 186 75 Prague 8, Czech Republic}


\date{\today}

\begin{abstract}
We use a method involving elementary submodels and a partial converse of Foran lemma to prove separable reduction theorems concerning Suslin $\sigma$-$\P$-porous sets where $\P$ can be from a rather wide class of porosity-like relations in complete metric spaces. In particular, we separably reduce the notion of Suslin cone small set in Asplund spaces. As an application we prove a theorem stating that a continuous approximately convex function on an Asplund space is Fr\'echet differentiable up to a cone small set.
\end{abstract}

\maketitle

\section{Introduction}
The present paper could be considered as a sequel to the articles \cite{Cuth} and \cite{Cuth-Rmoutil}.
Our main aim is to further investigate separable determination of various properties of sets and functions in metric spaces
(especially Banach spaces). That is, given a nonseparable metric space $X$ and a property of sets (or functions etc.) in $X$, we are interested whether certain statements about the property hold, provided that they hold in (some) separable subspaces of $X$.
In particular, if $X$ is a metric space, we are interested in $\sigma$-$\P$-porous sets in $X$, where $\P$ is a porosity-like relation on $X$
(see the definition below).

The key method we use to obtain separable determination results uses countable elementary structures which we call elementary submodels.
This method is described in Section~\ref{S:submodels}. Further details and examples of use can be found in \cite{Cuth} and \cite{Cuth-Rmoutil}.
However, the reader ought to note that there are other ways to go about this topic. An example is the use of rich families of Banach spaces,
which is described in detail, e.g., in \cite{LPT}. Sometimes one can also opt to prove this sort of results in the ``elementary way'',
in some sense replicating parts of the proof of the L\"owenheim-Skolem theorem. This approach would in many cases be very complicated, but can give a deeper insight, when possible.

In the rather technical Section~\ref{S:Foran} we prove essential auxiliary results which we use in Section~\ref{S:porosity-like} to prove our general separable determination result, Proposition \ref{P:reduction-porosity-like}. The general scheme of our proof is rather similar to that of separable determination of $\sigma$-upper porous sets in \cite{Cuth-Rmoutil} and involves the Foran lemma and it's partial converse.
The difference is that here we need no inscribing theorems as in \cite{Cuth-Rmoutil} to prove the statement for all Suslin sets.

Section~\ref{S:cone} contains the main result of this article, the separable determination of the notion of cone small sets in Asplund spaces
(Theorem \ref{T:cone}). Arguably the deeper part of this statement is a consequence of Proposition \ref{P:reduction-porosity-like}.
In the last section we provide several applications of our results, most notably Theorem \ref{T:appl-approx-convexity}.

Let us recall the most relevant notions, definitions, and notations. We denote by $\omega$ the set of all natural numbers (including $0$),
by $\en$ the set $\omega \setminus \{0\}$, by $\er_+$ the interval $(0,\infty)$, and $\qe_+$ stands for $\er_+ \cap \qe$.
We denote by ${\omega}^{<\omega}$ the set of all finite sequences of elements of $\omega$, and by $\omega^\omega$ the set of all infinite sequences. The notion of a countable set includes also finite sets. If $f$ is a mapping then we denote by $\rng f$ the range of $f$
and by $\dom f$ the domain of $f$. By writing $f\colon X \to Y$ we mean that $f$ is a mapping with $\dom f = X$
and $\rng f \subset Y$. By the symbol $f\rest_{Z}$ we denote the restriction of the mapping $f$ to the set $Z$.

If $(X,\rho)$ is a metric space, we denote by $U_X(x,r)$ the open ball, i.e., the set $\{y \in X\setsep \rho(x,y) < r\}$.
We often write $U(x,r)$ instead of $U_X(x,r)$. If $A,B \subset X$ are nonempty sets in a metric space $X$, we denote by $\d(A,B)$
the distance between the sets $A$ and $B$, i.e., $\d(A,B) = \inf\{\rho(a,b)\setsep a \in A, b \in B\}$.
We shall consider normed linear spaces over the field of real numbers. If $X$ is a normed linear space, $X^*$ stands for the dual space of $X$.

We say that $\P$ is a \emph{porosity-like relation} on the space $X$ if $\P$ is a relation between points $x \in X$ and sets $A \subset X$
(i.e. $\P \subset X \times 2^X$) satisfying the following conditions:
\begin{itemize}
\item[(i)]   if $A\subset B$ and $\P(x,B)$, then $\P(x,A)$,

\item[(ii)]  $\P(x,A)$ if and only if there is $r > 0$ such that $\P(x,A \cap U(x,r))$,

\item[(iii)] $\P(x,A)$ if and only if $\P(x,\overline{A})$.
\end{itemize}
A set is $\P$-porous, if $\P(x,A)$ for every $x \in A$. A set is $\sigma$-$\P$-porous,
if it is a countable union of $\P$-porous sets.

\section{Elementary submodels}\label{S:submodels}

In this section we recall some basic notions concerning the method of elementary submodels.
A brief description of this method can be found in \cite{Cuth-Rmoutil}, for a more detailed description see \cite{Cuth}.

First, let us recall some definitions. Let $N$ be a fixed set and $\phi$ a formula in the language of $ZFC$.
Then the \emph{relativization of $\phi$ to $N$} is the formula $\phi^N$ which is obtained from $\phi$ by replacing each quantifier of the form ``$\forall x$'' by ``$\forall x \in N$'' and each quantifier of the form ``$\exists x$'' by ``$\exists x \in N$''.
If $\phi(x_1,\ldots,x_n)$ is a formula with all free variables shown (i.e. a formula whose free variables are exactly $x_1,\ldots,x_n$) then
\emph{$\phi$ is absolute for $N$} if and only if
\[
\forall a_1, \ldots, a_n \in N\colon \bigl(\phi^N(a_1,\ldots,a_n) \leftrightarrow \phi(a_1,\ldots,a_n)\bigr).
\]
The method is based mainly on the following theorem (a proof can be found in \cite[Chapter IV, Theorem 7.8]{Kunen}).

\begin{theorem}\label{T:countable-model}
Let $\phi_1, \ldots, \phi_n$ be any formulas and $X$ any set. Then there exists a set $M \supset X$ such that
$\phi_1, \ldots, \phi_n \text{ are absolute for } M$ and $|M| \leq \max(\omega,|X|)$.
\end{theorem}

Since the set from Theorem \ref{T:countable-model} will often be used, the following notation is useful.

\begin{definition}
Let $\phi_1, \ldots, \phi_n$ be any formulas and $X$ be any countable set.
Let $M \supset X$ be a countable set such that $\phi_1, \ldots, \phi_n$ are absolute for $M$.
Then we say that $M$ is an \emph{elementary submodel for $\phi_1,\ldots,\phi_n$ containing $X$}.
This is denoted by $M \prec (\phi_1,\ldots,\phi_n; X)$.
\end{definition}

The fact that certain formula is absolute for $M$ will always be used in order to satisfy the assumption of the following lemma. It is a similar statement to \cite[Lemma 2.6]{Cuth}.
Using this lemma we can force the elementary submodel $M$ to contain all the needed objects created (uniquely) from elements of $M$.

\begin{lemma}\label{L:unique-M}
Let $\phi(y,x_1,\ldots,x_n)$ be a formula with all free variables shown and $Y$ be a countable set.
Let $M$ be a fixed set, $M \prec (\phi, \exists y \colon \phi(y,x_1,\ldots,x_n);\; Y)$, and
$a_1,\ldots,a_n \in M$ be such that there exists a set $u$ satisfying
$\phi(u,a_1,\ldots,a_n)$. Then there exists $u \in M$ such that $\phi(u,a_1,\ldots,a_n)$.
\end{lemma}

\begin{proof}
Using the absoluteness of the formula $\exists u\colon \phi(u,x_1,\ldots,x_n)$ there exists $u\in M$ satisfying $\phi^M(u,a_1,\ldots,a_n)$.
Using the absoluteness of $\phi$ we get, that for this $u\in M$ the formula $\phi(u,a_1,\ldots,a_n)$ holds.
\end{proof}

It would be very laborious and pointless to use only the basic language of the set theory.
For example, having a Banach space $X$ and $x, y \in X$ we often write $x + y$ and we know that this is a shortcut for a formula with free variables $x$, $y$, and $+$ (note that the symbol $+$ is here considered to be a variable as it depends on the Banach space).
Therefore, in the following text we use this extended language of the set theory as we are used to.
We shall also use the following convention.

\begin{convention}
Whenever we say ``\emph{for any suitable elementary submodel $M$ (the following holds \dots)}''
we mean that  ``\emph{there exists a list of formulas $\phi_1,\ldots,\phi_n$ and a countable set $Y$ such that for every $M \prec (\phi_1,\ldots,\phi_n;Y)$ (the following holds \dots)}''.
\end{convention}

By using this new terminology we lose the information about the formulas $\phi_1,\ldots,\phi_n$ and the set $Y$.
However, this is not important in applications.

Let us recall several further results about suitable elementary submodels
(all the proofs are based on Lemma \ref{L:unique-M} and they can be found in \cite[Chapters 2 and 3]{Cuth}).

\begin{proposition}
For any suitable elementary submodel $M$ the following holds.
\begin{itemize}
\item[(i)] If $A,B \in M$, then $A \cap B \in M$, $B \setminus A \in M$ and $A \cup B \in M$.

\item[(ii)] Let $f$ be a function such that $f\in M$. Then $\dom{f} \in M$, $\rng{f} \in M$ and for every
$x \in \dom{f} \cap M$ we have $f(x) \in M$.

\item[(iii)] Let $S$ be a finite set. Then $S \in M$ if and only if $S \subset M$.

\item[(iv)] Let $S \in M$ be a countable set. Then $S \subset M$.

\item[(v)] For every natural number $n > 0$ and for arbitrary $n+1$ sets $a_0, \ldots, a_n$ it is true that
\[
a_0, \ldots, a_n \in M \iff (a_0, \ldots, a_n) \in M.
\]
\end{itemize}
\end{proposition}

\begin{notation}\hfil
\begin{itemize}
\item If $A$ is a set, then by saying that an elementary model $M$ contains $A$ we mean that $A\in M$.

\item If $(X,\rho)$ is a metric space (resp. $(X, +, \cdot, \| \cdot \|)$ is a normed linear space)
and $M$ is an elementary submodel, then by saying {\em $M$ contains $X$} (or by writing $X\in M$) we mean that $(X,\rho) \in M$ (resp.
$(X, +, \cdot, \| \cdot \|) \in M$).

\item If $X$ is a topological space and $M$ is an elementary submodel, then we denote by $X_M$ the set $\ov{X\cap M}$.
\end{itemize}
\end{notation}

\begin{proposition}
For any suitable elementary submodel $M$ the following holds.
\begin{itemize}
\item[(i)] If $X$ is a metric space then whenever $M$ contains $X$, it is true that
\[
\forall r \in \er_+ \cap M\; \forall x \in X \cap M\colon U(x,r) \in M.
\]

\item[(ii)] If $X$ is a normed linear space then whenever $M$ contains $X$, it is true that $X_M$  is a closed separable subspace of $X$.
\end{itemize}
\end{proposition}

\begin{convention}
The proofs in the following text often begin in the same way.
To avoid unnecessary repetitions, by saying ``\emph{Let us fix a $(*)$-elementary submodel $M$ (containing $A_1,\ldots,A_n$)}''
we will understand the following.

\emph{Let us have formulas $\varphi_1,\ldots,\varphi_m$ and a countable set $Y$ such that the elementary submodel
$M \prec(\varphi_1,\ldots,\varphi_m; Y)$ is suitable for all the propositions from \cite{Cuth} and \cite{Cuth-Rmoutil}.
Add to them formulas marked with $(*)$ in all the preceding proofs from this paper and formulas marked with $(*)$ in the proof below and all their subformulas.
Denote such a list of formulas by $\psi_1,\ldots,\psi_k$.
Let us fix a countable set $X$ containing the sets $Y$, $\omega$, $\omega^\omega$, $\omega^{<\omega}$, $\zet$, $\qe$, $\qe_+$, $\er$, $\er_+$,
and all the common operations and relations on real numbers ($+$, $-$, $\cdot$, $:$, $<$).
Fix an elementary submodel $M$ for formulas $\psi_1,\ldots,\psi_k$ containing $X$ (such that $A_1, \ldots, A_n \in M$).}
\end{convention}

Thus, any $(*)$-elementary submodel $M$ is suitable for the results from \cite{Cuth}, \cite{Cuth-Rmoutil} and all the preceding theorems and propositions from this paper, making it possible to use all of these results for $M$.

In order to demonstrate how this technique works, we prove the following lemma which we use later.

\begin{lemma}\label{L:density}
For any suitable elementary submodel $M$ the following holds.
Let $(X,\rho)$ be a metric space and $\F$ be a countable collection of subsets of $X$.
Then whenever $M$ contains $X$ and $\F \subset M$, it is true that
\[
\bigcup \F \text{ is dense in } X \quad \Rightarrow \quad \bigcup \F \cap X_M \text{ is dense in } X_M.
\]
\end{lemma}

\begin{proof}
Let us fix a $(*)$-elementary submodel $M$ containing $X$ such that $\F \subset \mathcal{P}(X) \cap M$ and $\bigcup \F$ is dense in $X$.
In order to see that $\bigcup\F\cap X_M$ is dense in $X_M$, it is sufficient to prove that, for every $x\in X\cap M$ and $r \in \qe_+$,
there exists $F \in \F$ such that $U(x,r) \cap X_M \cap F \neq \emptyset$. Fix some $x \in X \cap M$ and $r \in \qe_+$.
Then there exists $F\in\F$ such that the following formula is satisfied
\begin{equation*}\tag{$*$}
\exists y \colon (y \in F \; \wedge \; \rho(x,y) < r).
\end{equation*}
The preceding formula has free variables $F$, $\rho$, $<$, $x$, and $r$.
Those are in $M$; hence, by Lemma \ref{L:unique-M}, there exists $y\in M$ such that $y \in F$ and $\rho(x,y) < r$.
Consequently, $U(x,r) \cap X_M \cap F \neq \emptyset$.
\end{proof}

\section{Foran scheme}\label{S:Foran}

We employ the following notation. Given $s,t \in \S$, we write $s \prec t$ if $t$ is an extension of $s$ (not necessarily proper).
We denote the concatenation of $s \in \S$ and $t \in \S$ by $s \h t$. If $s \in \S$ and $i \in \omega$, we write
$s\h i$ instead of $s\h (i)$. If $\nu = (\nu_0, \nu_1,\nu_2, \dots ) \in \omega^{\omega}$ and $n \in \omega$, then  the
symbol $\nu|n$ means the finite sequence $(\nu_0,\nu_1,\dots,\nu_{n-1})$. If $t \in \S$, then the symbol
$|t|$ denotes the length of $t$. By a \emph{tree} we mean any subset $T$ of $\S$ such that
for every $s \in \S$ and $t \in T$ with $s \prec t$, we have $s \in T$. We say that a tree $T$ is \emph{pruned} if for every $t \in T$
there exists $n \in \omega$ such that $t \h n \in T$.

Let us recall that any family $\mathcal A = \{A(s)\setsep s \in \S\}$ of sets is called a \emph{Suslin scheme} and
\emph{Suslin operation} $\mathcal S$ is defined by
\[
\mathcal S(\mathcal A) = \mathcal S_s(A(s)) = \bigcup_{\nu \in \omega^{\omega}}\bigcap_{n \in \omega} A(\nu|n).
\]
A Suslin scheme $\{A(s)\setsep s \in \S\}$ is called \emph{monotone} if $A(s) \supset A(t)$ whenever $s,t \in \S, s \prec t$.
Finally, a subset $A$ of a topological space $X$ is a \emph{Suslin} set (in $X$) if there exists a Suslin scheme $\mathcal A$
consisting of closed subsets of $X$ with $\mathcal S(\mathcal A) = A$.

\begin{setting}
Let $X$ be a complete metric space, $\P$ be a porosity-like relation on $X$, and let $\mathcal B$ be a basis of open sets in $X$.
\end{setting}

\begin{definition}
For any $A \subset X$ we define the following set operators:
\begin{align*}
\ker_{\P}(A) &= A \setminus \bigcup \{U \subset X \setsep U \text{ is open and $U \cap A$ is $\sigma$-$\P$-porous}\}, \\
N_{\P}(A)    &= \{x \in A\setsep \neg \P(x,A)\}.
\end{align*}
\end{definition}

The following lemma is easy to prove. Its assertions (i) and (ii) can be found, e.g., in \cite{Rmoutil}.

\begin{lemma}\label{L:basic}
Let $A \subset X$. Then we have
\begin{itemize}
\item[(i)]   $A \setminus \ker_{\P}(A)$ is $\sigma$-$\P$-porous,

\item[(ii)]  $\ker_{\P}(\ker_{\P}(A)) = \ker_{\P}(A)$,

\item[(iii)] if $A \subset X$ is a Suslin set then $\ker_{\P}(A)$ is a Suslin set,

\item[(iv)]  if $A \subset B \subset X$, $\ker_{\P}(B) = B$, and $B \setminus A$ is $\sigma$-$\P$-porous, then $\ker_{\P}(A) = A$,

\item[(v)]   $A \setminus N_{\P}(A)$ is $\P$-porous.
\end{itemize}
\end{lemma}

\begin{definition}\label{D:Foran-scheme}
Let $\F = \{S(t)\setsep \ t \in \omega^{<\omega}\}$ be a system of nonempty subsets of $X$
such that for each $t\in\omega^{<\omega}$ and each $k \in \omega$ we have
\begin{itemize}
\item[(i)] $\bigcup_{j\in\omega} S(t \h j)$ is a dense subset of $S(t)$,

\item[(ii)] $S(t)$ is $\P$-porous at no point of $S(t \h k)$,

\item[(iii)] for any $\nu\in \omega^\omega$ and any sequence $\{G_n\}_{n\in\omega}$ of sets from $\mathcal B$ satisfying:
\begin{itemize}
	\item[(a)] $\lim_{n \to \infty} \diam(G_n) = 0$,

	\item[(b)] $\overline{G_{n+1}} \subset G_n$ for every $n \in \omega$,

	\item[(c)] $S(\nu |n) \cap G_n \neq \emptyset$ for every $n \in \omega$,
\end{itemize}
we have
\[
\bigcap_{n\in\omega} \bigl(S(\nu |n) \cap G_n\bigr) \neq \emptyset.
\]
\end{itemize}
Then we say that $\F$ is a $(\mathcal B,\P)$\emph{-Foran scheme in} $X$.
If there is no danger of confusion we will say just \emph{Foran scheme}.
\end{definition}	

\begin{lemma}\label{L:Foran-lemma}
Let $\F$ be a Foran scheme in $X$. Then each element of $\F$ is a non-$\sigma$-$\P$-porous set.
\end{lemma}

\begin{proof}
We mimic the standard proof which works for Foran systems, see \cite[Lemma 4.3]{Zajicek-3}.
It is sufficient to prove that $S(\emptyset)$ is non-$\sigma$-$\P$-porous.
Suppose on the contrary that $S(\emptyset) = \bigcup_{n=1}^{\infty} A_n$, where each $A_n$ is $\P$-porous.
We set $A_0 = \emptyset$.
We will construct $\nu = (\nu_0,\nu_1,\dots) \in \omega^{\omega}$ and a sequence of open sets $\{G_n\}_{n \in \omega}$ such that
for every $n \in \omega$ we have
\begin{itemize}
	\item[(a)] $\diam(G_n) < 2^{-n}$,
	\item[(b)] $\overline{G_{n+1}} \subset G_n$,
	\item[(c)] $G_n \cap S(\nu |n) \neq \emptyset$,
    \item[(d)] $S(\nu|n) \cap G_n \cap A_n = \emptyset$,
    \item[(e)] $G_n \in \mathcal B$.
\end{itemize}
We will construct inductively $\nu_n$'s and $G_n$'s.
If $n = 0$, then we find an open set $G_0 \in \mathcal B$ intersecting $S(\emptyset)$ with  $\diam G_0 < 1$.
Then conditions (a)--(e) are clearly satisfied.
Now suppose that we have already constructed $G_n$ and $s = (\nu_0,\dots,\nu_{n-1})$ for $n \in \omega$.
We distinguish two cases.

If $A_{n+1}$ is not dense in $S(s) \cap G_n$, then we find a nonempty open set $G_{n+1} \in \mathcal B$ such that $G_{n+1} \cap S(s) \neq \emptyset$,
$\overline{G_{n+1}} \subset G_n \setminus A_{n+1}$, and $\diam G_{n+1} < 2^{-(n+1)}$. Further, using condition (i) from Definition~\ref{D:Foran-scheme} we find $\nu_n \in \omega$
such that $S(s \h \nu_n) \cap S(s) \cap G_{n+1} \neq \emptyset$.

Now suppose that $A_{n+1}$ is dense in $S(s) \cap G_n$. Then we find $\nu_n$
such that $S(s \h \nu_n) \cap  S(s) \cap G_n \neq \emptyset$ by condition (i) from Definition~\ref{D:Foran-scheme}.
Suppose that $x \in A_{n+1} \cap S(s) \cap G_n$. Then we have $\P(x,A_{n+1})$ and by density of $A_{n+1}$ in $S(s) \cap G_n$ we get
$\P(x,S(s) \cap G_n)$ since $\P$ is a porosity-like relation. It implies $x \notin S(s \h \nu_n)$
by condition (ii) from Definition~\ref{D:Foran-scheme}. Thus we get that $S(s \h \nu_n) \cap G_n \cap A_{n+1} = \emptyset$.
We find an open set $G_{n+1}  \in \mathcal B$ such that $\overline{G_{n+1}} \subset G_n$, $G_{n+1} \cap S(s \h \nu_n) \neq \emptyset$, and
$\diam G_{n+1} < 2^{-(n+1)}$. This finishes the construction of $\nu$ and $G_n$'s.

Since $\F$ is a Foran scheme there exists $x \in \bigcap_{n \in \omega} \bigl(S(\nu|n) \cap G_n\bigr)$.
By (b) and (d) we have $x \in S(\emptyset) \setminus \bigcup_{n=1}^{\infty} A_n = \emptyset$, a contradiction.
\end{proof}

\begin{definition}
A Suslin scheme $\mathcal C = \{C(s) \setsep s \in \S\}$ is \emph{subordinate} to a Suslin scheme $\mathcal A = \{A(s)\setsep  s \in \S\}$
(notation $\mathcal C \sqsubseteq \mathcal A$) if there exists a mapping $\varphi\colon \S \to \S$ such that for each $s \in \S$ we have
\begin{itemize}
\item $|\varphi(s)| = |s|$,

\item if $t \in \S, s \prec t$, then $\varphi(s) \prec \varphi(t)$,

\item $C(s) \subset A(\varphi(s))$.
\end{itemize}
\end{definition}

\begin{definition}\label{D:regular}
Let $\mathcal A = \{A(s)\setsep s \in \S\}$ be a Suslin scheme. Denote $C(s) = \mathcal S_t(A (s \h t))$, $s \in \S$.
We say that $\mathcal A$ is \emph{$\P$-regular} if $\mathcal A$ is monotone and
for every $s \in \S$ we have $\ker_{\P} (C(s)) = C(s) \neq \emptyset$.
\end{definition}

\begin{lemma}\label{L:sub}
Let $\mathcal A$ be a Suslin scheme consisting of closed sets and $C \subset \mathcal S(\mathcal A)$ be a Suslin set.
Then there exists a Suslin scheme $\mathcal C$ consisting of closed sets which is subordinate to $\mathcal A$ and $C = \mathcal S(\mathcal C)$.
\end{lemma}

\begin{proof}
Let $\mathcal L = \{L(s)\setsep s \in \S\}$ be a Suslin scheme consisting of closed sets with $\mathcal S(\mathcal L) = C$.
Let $\mathcal A = \{A(s)\setsep s \in \S \}$. Fix a bijection $\psi = (\psi_1,\psi_2) \colon \omega \to \omega^2$.
We define mappings $\varphi\colon \S \to \S$, $\rho\colon \S \to \S$ by $\varphi(\emptyset) = \rho(\emptyset) = \emptyset$ and
\begin{align*}
\varphi(s) &= \psi_1(s_0) \h \psi_1(s_1) \h \dots \h \psi_1(s_{|s|-1}), \\
\rho(s) &= \psi_2(s_0) \h \psi_2(s_1) \h \dots \h \psi_2(s_{|s|-1}),
\end{align*}
where $s = (s_0, \dots, s_{|s|-1}) \in \S \setminus \{\emptyset\}$.
We define the desired scheme $\mathcal C$ by $C(s) = A(\varphi(s)) \cap L(\rho(s))$.
The scheme $\mathcal C = \{C(s)\setsep s \in \S\}$ consists of closed sets and is clearly subordinate to $\mathcal A$
via the mapping  $\varphi$.

We verify the equality $C = \mathcal S(\mathcal C)$. If $x \in C$, then there exist $\nu,\mu \in \omega^{\omega}$ such that $x \in L(\nu|k)$ and $x \in A(\mu|k)$ for every $k \in \omega$.
Since $\psi$ is a bijection of $\omega$ onto $\omega^2$ we can find $\tau \in \omega^{\omega}$ such that $\varphi(\tau|k) = \nu|k$ and $\rho(\tau|k) = \mu|k$ for every $k \in \omega$.
Then we have $x \in C(\tau|k)$ for every $k \in \omega$. Consequently, $x \in \mathcal S(\mathcal C)$.

If $x \in \mathcal S(\mathcal C)$, then there exists $\tau \in \omega^{\omega}$ such that $x \in C(\tau|k)$ for every $k \in \omega$.
We find $\mu,\nu \in \omega^{\omega}$ such that $\varphi(\tau|k) = \mu|k$ and  $\rho(\tau|k) = \nu|k$ for every $k \in \omega$.
Then we have $x \in A(\mu|k) \cap L(\nu|k)$. Consequently, $x \in C$. Thus we have proved $C = \mathcal S(C)$.
\end{proof}

\begin{lemma}\label{L:regular}
Let $\mathcal A$ be a Suslin scheme consisting of closed subsets of $X$
and $C \subset \mathcal S(\mathcal A)$ be a Suslin set with $\ker_{\P}(C) = C \neq \emptyset$.
Then there exists a $\P$-regular Suslin scheme $\mathcal L = \{L(s) \setsep s \in \S\}$ consisting of closed subsets of $X$
such that $\mathcal L$ is subordinate to $\mathcal A$ and $\mathcal S(\mathcal L)$ is a dense subset of $C$.
\end{lemma}

\begin{proof}
Let $\mathcal A = \{A(s)\setsep s \in \S\}$.  Using Lemma~\ref{L:sub} we find a Suslin scheme $\mathcal D = \{D(s)\setsep s \in \S\}$
consisting of closed subsets of $X$ which is subordinate to $\mathcal A$ and $\mathcal S(\mathcal D) = C$.
For $s \in \S$ we set
\[
E(s) = \mathcal S_t(D(s \h t)), \qquad
H(s) = \ker_{\P}(E(s)), \qquad
P(s) = \overline{H(s)}, \qquad
\text{and} \qquad Q(s) = \mathcal S_t(P(s \h t)).
\]
We verify that $\ker_{\P}(Q(s)) =  Q(s)$ for every $s \in \S$. For every $u \in \S$ we have
\[
E(u) = \bigcup_{j \in \omega} E(u \h j)
\]
and
\[
H(u) \setminus \bigcup_{j \in \omega} H(u \h j) \subset \Bigl(E(u) \setminus \bigcup_{j \in \omega} E(u \h j)\Bigr) \cup
\bigcup_{j \in \omega} \bigl(E(u \h j)\setminus H(u \h j)\bigr)
= \bigcup_{j \in \omega} \bigl(E(u \h j)\setminus H(u \h j)\bigr).
\]
Since $E(u \h j) \setminus H(u \h j)$ is $\sigma$-$\P$-porous for every $j \in \omega$ (Lemma~\ref{L:basic}(i)), we conclude that
the set $H(u) \setminus \bigcup_{j \in \omega} H(u \h j)$ is $\sigma$-$\P$-porous.
Since for every $s \in \S$ we have
\[
H(s) \setminus \mathcal S_t(H(s\h t)) \subset  \bigcup_{t\in\S} \Bigl(H(s\h t) \setminus \bigcup_{j\in \omega} H(s\h t \h j)\Bigr),
\]
we get that  $H(s) \setminus \mathcal S_t(H(s\h t))$ is $\sigma$-$\P$-porous. Therefore
\begin{equation}\label{E:ker-S-t}
\ker_{\P}(\mathcal S_t(H(s\h t))) = \mathcal S_t(H(s\h t))
\end{equation}
by Lemma~\ref{L:basic}(iv). Further, we get that $\mathcal S_t(H(s\h t))$ is a dense subset of $H(s)$, thus $\overline{\mathcal S_t(H(s\h t))} = \overline{H(s)}$.
Observing
\begin{equation}\label{E:inclusions}
\mathcal S_t(H(s\h t)) \subset \mathcal S_t(P(s\h t)) = Q(s) \subset P(s) = \overline{H(s)} = \overline{\mathcal S_t(H(s\h t))}
\end{equation}
and using \eqref{E:ker-S-t} we get $\ker_{\P}(Q(s)) = Q(s)$. Indeed, fix an open set $U$ with $U\cap Q(s)\neq\emptyset$. Using \eqref{E:inclusions}, $U\cap S_t(H(s\h t))\neq\emptyset$ and, by \eqref{E:ker-S-t}, $U\cap S_t(H(s\h t))$ is not $\sigma$-$\P$-porous set. Hence, $U\cap Q(s)$ is not $\sigma$-$\P$-porous set and, as $U$ was an arbitrary open set intersecting $Q(s)$, $\ker_{\P}(Q(s)) = Q(s)$.

Further, we set $T = \{s \in \S\setsep P(s) \neq \emptyset\}$. The set $T$ is obviously a nonempty tree. Moreover, $T$ is pruned. Indeed,
if $s \in T$, then $H(s) \neq \emptyset$ and thus $E(s)$ is non-$\sigma$-$\P$-porous.
We have $E(s) = \bigcup_{n \in \omega} E(s \h n)$ and therefore there exists $m \in \omega$ such that $E(s \h m)$ is non-$\sigma$-$\P$-porous. Thus $P(s \h m) \neq \emptyset$ and $s\h m \in T$.

We find a mapping $\varphi \colon \S \to \S$ such that for every $s \in \S$ we have
\begin{itemize}
\item $|\varphi(s)| = |s|$,

\item if $t \in \S, s \prec t$, then $\varphi(s) \prec \varphi(t)$,

\item $\{\varphi(s \h n)\setsep n \in \omega\} = \{\varphi(s) \h k \setsep k \in \omega\} \cap T$.
\end{itemize}
We have $\emptyset \in T$ since $T$ is nonempty.
We set $\varphi(\emptyset) = \emptyset$. Suppose that $\varphi(s) \in T$ has been already defined for some $s \in \S$.
The set $W := \{k \in \omega \setsep \varphi(s) \h k \in T\}$ is nonempty since $T$ is pruned.
Thus we can find a mapping $\psi\colon \omega \to \omega$ such that
$\psi(\omega) = W$. We define $\varphi(s \h n) = \varphi(s) \h \psi(n)$.
This finishes the construction of $\varphi$. It is easy to check that the mapping $\varphi$ has all the required properties.

We set $L(s) = P(\varphi(s))$ and $\mathcal L = \{L(s) \setsep s \in \S\}$.
The scheme $\{E(s)\setsep s \in \S\}$ is monotone by definition.
This easily gives that the scheme $\mathcal L$ is also monotone.

By the properties of $\varphi$ and the definition of $T$
we have $\mathcal S_t(L(s \h t)) = Q(\varphi(s)) \neq \emptyset$ for every $s \in \S$.
Indeed, if $x \in Q(\varphi(s))$ for some $s \in \S$, then there exists $\nu \in \omega^{\omega}$ such that
$x \in P(\varphi(s) \h \nu|n)$ for  every $n \in \omega$. Thus $P(\varphi(s) \h \nu|n) \neq \emptyset$
for every $n \in \omega$. This means that $\varphi(s) \h \nu|n \in T$ for every $n \in \omega$.
Using the properties of $\varphi$ we find $\mu \in \omega^{\omega}$ such that $\varphi(s \h \mu|n) = \varphi(s) \h \nu|n$ for every $n \in \omega$.
Thus we have $x \in \bigcap_{n\in\omega}P(\varphi(s) \h \nu|n) = \bigcap_{n\in\omega} L(s \h \mu|n) \subset \mathcal S_t(L(s \h t))$.
If $x \in \mathcal S_t(L(s \h t))$, then there exists $\nu \in \omega^{\omega}$ such that
$x \in L(s \h \nu|n) = P(\varphi(s \h \nu|n))$ for every $n \in \omega$. Using the properties of $\varphi$ again, we get
$x \in \bigcap_{n \in \omega} P(\varphi(s \h \nu|n)) \subset Q(\varphi(s))$.
Finally, if $s \in \S$ we have $H(\varphi(s)) \neq \emptyset$ and $H(\varphi(s)) \setminus \mathcal S_t(H(\varphi(s) \h t))$ is $\sigma$-$\P$-porous.
Thus $\mathcal S_t(H(\varphi(s) \h t)) \neq \emptyset$ and by \eqref{E:inclusions} we get $Q(\varphi(s)) \neq \emptyset$.
Thus $\mathcal L$ is $\P$-regular.

Clearly $\mathcal L \sqsubseteq \mathcal D$. Using the fact that $\mathcal D \sqsubseteq \mathcal A$, we get
$\mathcal L \sqsubseteq \mathcal A$. It remains to verify that $\mathcal S(\mathcal L)$ is dense in $C$.
Since by definition we have $P(s) \subset D(s)$ for every $s \in \S$, we get $Q(\emptyset) \subset E(\emptyset) = C$.
The set $\mathcal S_t(H(t))$ is a dense subset of $H(\emptyset)$. We get by \eqref{E:inclusions} that $Q(\emptyset)$ is a dense subset of $H(\emptyset) = C$. It concludes the proof since $\mathcal S(\mathcal L) = Q(\emptyset)$.
\end{proof}

\begin{proposition}\label{P:Foran-converse}
Suppose that $N_{\P}(A)$ is a Suslin set whenever $A \subset X$ is Suslin.
If $S \subset X$ is a Suslin non-$\sigma$-$\P$-porous set, then there exists a $(\mathcal B,\P)$-Foran scheme $\F$ in $X$
such that each element of $\F$ is a subset of $S$.
\end{proposition}

\begin{proof}
We will construct a sequence $(\mathcal A_n)_{n \in \omega}$ of $\P$-regular Suslin schemes consisting of closed sets,
where $\mathcal A_n = \{A^n(s)\setsep s \in \S\}$. For $s \in \S$ we denote $C^n(s) = \mathcal S_t (A^n(s \h t))$ and, for every $n \in \omega$, we  require
\begin{itemize}
\item $\mathcal S (\mathcal A_0) \subset S$,

\item $\mathcal A_{n+1} \sqsubseteq \mathcal A_n$ and this fact is witnessed by a mapping $\varphi_{n+1}\colon \S \to \S$,

\item $\varphi_{n+1}(s) = s$ for every $s \in \S, |s| \leq n$.
\end{itemize}

Using Lemma~\ref{L:regular} we find a $\P$-regular Suslin scheme $\mathcal A_0$ consisting of closed sets
with $\mathcal S (\mathcal A_0) \subset \ker_{\P}(S) \subset S$.
Suppose that $n \in \omega, n > 0$, and we have already constructed the desired schemes $\mathcal A_j$, $j < n$.
Fix $s \in \omega^{n-1}$. The set $\ker_{\P}(N_{\P}(C^{n-1}(s)))$ is nonempty by $\P$-regularity of $\mathcal A_{n-1}$.
This set is also Suslin by the assumption and Lemma~\ref{L:basic}(iii).
By Lemma~\ref{L:regular} we find a $\P$-regular Suslin scheme $\mathcal L_s = \{L^s(t) \setsep t \in \S\}$
such that $\mathcal S(\mathcal L_s)$ is a dense subset of $\ker_{\P}(N_{\P}(C^{n-1}(s)))$
and $\mathcal L_s \sqsubseteq \{A^{n-1}(s \h t)\setsep t \in \S\}$. The last fact is witnessed by a mapping $\varphi^s_n \colon \S \to \S$.

For $t = (t_0,\dots,t_{|t|-1}) \in \S$, we set
\begin{align*}
A^{n}(t) &=
\begin{cases}
A^{n-1}(t), & |t| < n, \\
L^{t|(n-1)}(t_{n-1},\dots,t_{|t|-1}), & |t| \geq n;
\end{cases}
\\
\varphi_n(t) &=
\begin{cases}
t, & |t| < n, \\
t|(n-1) \h \varphi_n^{t|(n-1)}(t_{n-1},\dots,t_{|t|-1}), & |t| \geq n.
\end{cases}
\end{align*}
Further, we set $\mathcal A_n = \{A^n(t)\setsep t \in \S\}$.

We define $S(s) = C^{|s|}(s)$, $s \in \S$, and $\F = \{S(s)\setsep s \in \S\}$.
We verify the conditions (i)--(iii) from Definition~\ref{D:Foran-scheme}.

(i) Let $n \in \omega$ and $s \in \omega^n$. Since $\ker_{\P}(C^n(s)) = C^n(s)$
we have that the set $\ker_{\P}(N_{\P}(C^n(s)))$ is dense in $C^n(s)$.
Since $\mathcal S(\mathcal L_s)$ is dense in $\ker_{\P}(N_{\P}(C^n(s)))$, we get that $\mathcal S(\mathcal L_s)$ is dense
in $C^n(s)$. Thus we have that $\bigcup_{j \in \omega} S(s \h j) = \bigcup_{j \in \omega} C^{n+1}(s \h j) = C^{n+1}(s) = \mathcal S(\mathcal L_s)$
is a dense subset of $S(s) = C^n(s)$.

(ii) We have $S(t \h k) = C^{|t|+1}(t \h k)  \subset N_{\P}(C^{|t|}(t)) = N_{\P}(S(t))$ for every $t \in \S$ and $k \in \omega$.

(iii) Suppose that we have $\nu \in \omega^{\omega}$ and a sequence $\{G_n\}_{n\in\omega}$ of open sets such that
\begin{itemize}
	\item[(a)] $\lim_{n \to \infty} \diam(G_n) = 0$,
	\item[(b)] $\overline{G_{n+1}} \subset G_n$ for every $n \in \omega$,
	\item[(c)] $S(\nu |n) \cap G_n \neq \emptyset$ for every $n \in \omega$.
\end{itemize}

We have that $\bigcap_{n \in \omega} G_n = \{x\}$ for some $x \in X$, since $X$ is complete.
Our task is to show that $x \in S(\nu|m)$ for every $m \in \omega$.
Fix $m \in \omega$. For each $k \in \omega$ there exists $y_k \in S(\nu|k) \cap G_k$. We have $\lim y_k = x$. Fix $k \in \omega$, $k \geq m$.
For every $n \in \omega, n > k$, we have
\begin{equation}\label{E:inclusions2}
y_n \in S(\nu|n) = C^n(\nu|n) \subset A^n(\nu|n) \subset A^{n-1}(\varphi_n(\nu|n)) \subset \dots \subset A^m(\varphi_{m+1} \circ \dots \circ \varphi_n(\nu|n)).
\end{equation}
Since $\mathcal A_m$ is a $\P$-regular scheme, it is monotone. Using this fact and \eqref{E:inclusions2} we get
\begin{equation}\label{E:inclusions3}
y_n \in A^m(\varphi_{m+1} \circ \dots \circ \varphi_n(\nu|n))
\subset A^m(\varphi_{m+1} \circ \dots \circ \varphi_n(\nu|k)).
\end{equation}
Since $\varphi_j(\nu|k) = \nu|k$ for every $j \in \omega, j > k$, we get
\[
A^m(\varphi_{m+1} \circ \dots \circ \varphi_n(\nu|k)) = A^m(\varphi_{m+1} \circ \dots \circ \varphi_{k+1}(\nu|k)).
\]
Using this and \eqref{E:inclusions3} we get $x \in A^m(\varphi_{m+1} \circ \dots \circ \varphi_{k+1}(\nu|k))$
since the latter set is closed. Since $\nu|m \prec \varphi_{m+1} \circ \dots \circ \varphi_{k+1}(\nu|k)$ we can conclude $x \in S(\nu|m)$.
\end{proof}

\section{Porosity-like relations}\label{S:porosity-like}

\begin{definition}
Let $X$ be a metric space and ${\bf R}$ be a point-set relation on $X$ (i.e. ${\bf R}\subseteq X\times \mathcal{P}(X)$).
Let $M$ be a set and ${\bf R}'$ be a point-set relation on $X_M = \ov{X\cap M}$.
We say that the set $M$ is a \emph{pointwise $({\bf R}\to {\bf R}')$-model} if
\[
\forall A \in \mathcal{P}(X) \cap M \ \forall x \in X_M \colon  \bigl({\bf R}(x,A)\;\to\; {\bf R}'(x,A\cap X_M)\bigr).
\]
Similarly, we define the notion of a \emph{pointwise $({\bf R}\leftrightarrow {\bf R}')$-model}.
\end{definition}

\begin{definition}
Let $X$ be a metric space and $\P$ be a porosity-like relation on $X$.
Let $M$ be a set and $\P'$ be a porosity-like relation on $X_M = \ov{X\cap M}$.
We say that the set $M$ is a \emph{$(\P \to \P')$-model} if for every set $A\in \mathcal{P}(X)\cap M$
\[
A \text{ is $\P$-porous in the space } X\;\to\;A \cap X_M \text{ is $\P'$-porous in the space }X_M.
\]
We say that the set $M$ is a \emph{$(\sigma$-$\P \to \sigma$-$\P')$-model} if for every set $A\in \mathcal{P}(X)\cap M$
\[
A \text{ is $\sigma$-$\P$-porous in the space } X\;\to\;A \cap X_M\text{ is $\sigma$-$\P'$-porous in the space }X_M.
\]
Similarly, we define the notion of \emph{$(\P \leftarrow \P')$-model}, \emph{$(\P \leftrightarrow \P')$-model},
and \emph{$(\sigma$-$\P \leftrightarrow \sigma$-$\P')$-model}.
\end{definition}

\begin{proposition}\label{P:porosity-like}
For any suitable elementary submodel $M$ the following holds.
Let $X$ be a metric space, $\P$ be a porosity-like relation on $X$ and $\P'$ be a porosity-like relation on $X_M$.
Assume $M$ contains $X$ and $\P$.
\begin{itemize}
\item[(i)]   If $M$ is a pointwise $(\P \to \P')$-model, then $M$ is a $(\P \rightarrow \P')$-model.
\item[(ii)]  If $M$ is a pointwise $(\neg \P \to \neg \P')$-model, then $M$ is a $(\P \leftarrow \P')$-model.
\item[(iii)] If $M$ is a $(\P \to \P')$-model, then $M$ is a $(\sigma$-$\P \to \sigma$-$\P')$-model.
\end{itemize}
In particular, if $M$ is a pointwise $(\P \leftrightarrow \P')$-model, then $M$ is a $(\P  \leftrightarrow \P')$-model and a $(\sigma$-$\P \to \sigma$-$\P')$-model.
\end{proposition}

\begin{proof}
Let us fix a $(*)$-elementary submodel $M$ containing $X$ and $\P$.

(i) The statement follows immediately from definitions.

(ii) Let us suppose $M$ is a pointwise $(\neg \P \to \neg \P')$-model and let us fix a non-$\P$-porous set $A\in\mathcal{P}(X)\cap M$.
Consider the formula
\begin{equation*}\tag{$*$}
\exists x\in A \colon (x,A) \notin \P,
\end{equation*}
with free variables $A$ and $\P$.
Since $A \in M$, $\P \in M$, and the above formula is absolute for $M$, there exists $x \in A \cap M$ such that $(x,A) \notin \P$, i.e.,
$A$ is not $\P$-porous at $x$.
Hence $A\cap X_M$ is not $\P'$-porous at $x$. Thus, $A\cap X_M$ is not $\P'$-porous in the space $X_M$ and (ii) holds.

(iii) Suppose that $A \in M \cap \mathcal P(X)$ is $\sigma$-$\P$-porous. Then the next formula is satisfied
\begin{equation*}\tag{$*$}
\begin{split}
\exists D \colon \bigl(& D \text{ is a function with }\dom D = \en, D(n) \subset X \text{ is $\P$-porous set}\\
& \text{for every }n \in \en,\;  A \subset \bigcup_{n \in \en} D(n)\bigr).
\end{split}
\end{equation*}
Now by Lemma~\ref{L:unique-M} we find $D \in M$ such that
\begin{equation*}
\begin{split}
& D \text{ is a function with }\dom D = \en, D(n) \subset X \text{ is $\P$-porous set}\\
& \text{for every }n \in \en,\;  A \subset \bigcup_{n \in \en} D(n).
\end{split}
\end{equation*}
We have $D(n) \in M$ for every $n \in \en$. Since $M$ is a $(\P \to \P')$-model, we obtain that $D(n) \cap X_M$ is $\P'$-porous in $X_M$ for every $n \in \en$,
hence, $A \cap X_M$ is $\sigma$-$\P'$-porous in $X_M$.
\end{proof}

\begin{lemma}\label{L:metric-lemma}
Let $(X,\varrho)$ be a metric space and $(Y,\varrho)$ be its complete subspace.
Assume we have a sequence $U_Y^n = U_Y(x_n,r_n) \subseteq Y$, $n \in \en$, of open balls in $Y$ and assume the following conditions hold:
\begin{itemize}
\item[(i)] $\lim_{n\to \infty}r_n = 0$,
\item[(ii)] for each $n\in \en$ we have $\ov{U_Y^{n+1}\;}^Y \subseteq U_Y^n$.
\end{itemize}
Denote $U_X^n := U_X(x_n,r_n)$. Then there exists an increasing sequence of integers $\{ n(k) \}_{k=1}^\infty$ such that for each $k\in \en$ we have
\[
\ov{U_X^{n(k+1)}\;}^X \subseteq U_X^{n(k)}.
\]
\end{lemma}

\begin{proof}
We shall prove the following statement which implies the conclusion of the lemma: For~each $k\in\en$ there is $l\in\en$, $l>k$ such that $\ov{U_X^l\;}^X\subseteq U_X^k$.

Assume this is not the case, i.e., there is a natural number $n_0$ such that $\ov{U_X^n\;}^X \setminus U_X^{n_0} \neq \emptyset$ for each natural number $n>n_0$. Choose a sequence $\{y_n\}_{n=n_0+1}^\infty$ such that $y_n\in\ov{U_X^n\;}^X \setminus U_X^{n_0}$ for each $n>n_0$.

From the assumptions it is obvious that the sequence $\{x_n\}_{n=1}^\infty$ is Cauchy and hence it has a limit $x\in Y$ (as $Y$ is complete).
Since $\varrho (x_n,y_n)\leq r_n$, it also follows from (i) that $\lim_{n\to\infty}y_n=x$. Consequently, $x\notin U_X^{n_0}$ as $U_X^{n_0}$ is open and $y_n \notin U_X^{n_0}$ for any $n > n_0$.

On the other hand, the assumption (ii) gives that $\{ x_n\setsep n>n_0\}\subseteq U_Y^{n_0+1}$ and so, again by (ii), $x = \lim_{n\to\infty} x_n \in \ov{U_Y^{n_0+1}\;}^Y \subseteq U_Y^{n_0}$.
This is a contradiction.
\end{proof}

\begin{proposition}\label{P:reduction-Foran}
For any suitable elementary submodel $M$ the following holds.
Let $X$ be a complete metric space, $A\subset X$ and $\P$ be a porosity-like relation on $X$.
Let there exist a $(\B, \P)$-Foran scheme $\F$ in $X$, where $\B = \{U(x,r)\setsep x \in X,\; r \in \er_+\}$,
such that each element of $\F$ is a subset of $A$.
Assume that $M$ contains $X$, $A$, $\P$ and $M$ is a pointwise $(\neg \P \to \neg \P')$-model for some porosity-like relation $\P'$ on $X_M$.

Then there exists a $(\B', \P')$-Foran scheme $\F'$ in $X_M$,
where $\B' = \{U(x,r)\cap X_M \setsep x \in X \cap M,\; r \in \qe_+\}$,
such that each element of $\F'$ is a subset of $A \cap X_M$.
\end{proposition}

\begin{proof}
By the assumption, the following formula is true:
\begin{equation*}\tag{$*$}
\begin{split}
\exists S\;(S:\omega^{<\omega}\to\mathcal{P}(X)\text{ is such that } \{S(t)\setsep t\in\omega^{<\omega}\}&\text{ is a }(\B,\P)\text{-Foran scheme in }X\\
& \text{and, for every $t\in\omega^{<\omega}$, $S(t)\subset A$}).
\end{split}
\end{equation*}
Using Lemma~\ref{L:unique-M} and absoluteness of the preceding formula and its subformula for $M$, we find the corresponding $S \in M$.
Consequently, for every $t \in \omega^{<\omega}$, $S(t) \in M$. Now it is sufficient to prove, that
\[
\F' = \{S(t)\cap X_M \setsep t \in \omega^{<\omega}\}
\]
is a $(\B',\P')$-Foran scheme in $X_M$.
By \cite[Lemma 2.10]{Cuth-Rmoutil}, $S(t)\cap M$ is a dense subset of $S(t)\cap X_M$ for every $t \in \omega^{<\omega}$.
Hence, by Lemma \ref{L:density}, the condition (i) from Definition \ref{D:Foran-scheme} is satisfied.
Since $M$ is a pointwise $(\neg \P \to \neg \P')$-model for $\P'$ on $X_M$ we get by Proposition~\ref{P:porosity-like}(ii)
that the condition (ii) is satisfied.

In order to prove that (iii) holds, let us take some $\nu\in \omega^\omega$, a sequence $\{x_n\}_{n\in\omega}$ of elements of $X\cap M$ and a sequence $\{r_n\}_{n\in\omega}$ of numbers from $\qe_+$ such that the open balls $G_n = U(x_n,r_n)\cap X_M$ satisfy conditions (a), (b), and (c) in the space $X_M$. It is easy to see that the radii $r_n$ can be chosen in such a way that $r_n<2\diam(U(x_n,r_n)\cap X_M)=2\diam(G_n)$, and hence $r_n\to 0$. Then Lemma \ref{L:metric-lemma} gives the existence of an increasing sequence of integers $\{n(k)\}_{k=1}^\infty$ such that $\ov{U(x_{n(k+1)},r_{n(k+1)})\;}^X \subseteq U(x_{n(k)},r_{n(k)})$ for each $k$. Hence we have that the sequence $\left\{U(x_{n(k)},r_{n(k)}) \right\}_{k=1}^\infty$ satisfies condition (b) from the definition of $(\B, \P)$-Foran scheme in $X$ and the condition (a) follows from our assumption that $r_n\to 0$.
Now we verify the condition (c).
From the assumptions on $G_n$ we know that
$U(x_{n(k)},r_{n(k)}) \cap S(\nu \rest k) \supseteq U(x_{n(k)},r_{n(k)}) \cap S(\nu\rest n(k)) \cap X_M \neq \emptyset$
and so (as $\F$ is a $(\B, \P)$-Foran scheme in $X$) we have that there exists $x \in \bigcap_{k=1}^\infty U(x_{n(k)},r_{n(k)}) \cap S(\nu\rest k)$.
Since $\lim x_n = x$ by (a) and (b), we have $x \in X_M$. Consequently,
\[
x \in  \bigcap_{n=1}^\infty(G_n \cap S(\nu\rest n)\cap X_M).
\]
This verifies (iii) from Definition~\ref{D:Foran-scheme}.
\end{proof}

\begin{proposition}\label{P:reduction-porosity-like}
For any suitable elementary submodel $M$ the following holds.
Let $X$ be a complete metric space, $\P$ be a porosity-like relation on $X$, and ${\P'}$ be a porosity-like relation on $X_M$. Suppose that $N_{\P}(S)$ is a Suslin set whenever $S\subset X$ is Suslin. Assume $M$ contains $X$, $\P$, and a Suslin set $A\subset X$. Then whenever $M$ is a pointwise $(\P \leftrightarrow \P')$-model, then the following holds:
\begin{align*}
A \text{ is $\P$-porous in the space }X            &\iff A \cap X_M \text{ is $\P'$-porous in the space } X_M, \\
A \text{ is $\sigma$-$\P$-porous in the space }X   &\iff A \cap X_M \text{ is $\sigma$-$\P'$-porous in the space } X_M.
\end{align*}
\end{proposition}

\begin{proof}
By Proposition \ref{P:porosity-like}, it is sufficient to prove the implication from the right to the left in the second equivalence.
Let us fix a $(*)$-elementary submodel $M$ containing $X$, $\P$, and a non-$\sigma$-$\P$-porous Suslin set $A\subset X$.
We would like to verify that $A$ is not $\sigma$-$\P'$-porous in the space $X_M$. Let $\B$, $\B'$ be as in Proposition \ref{P:reduction-Foran}.
By Proposition \ref{P:Foran-converse}, there exists a $(\B,\P)$-Foran scheme $\F$ in $X$ such that each element of $\F$ is a subset of $A$.
Using Proposition \ref{P:reduction-Foran}, there exists a $(\B',\P')$-Foran scheme $\F'$ in $X_M$ such that each element of $\F'$
is a subset of $A \cap X_M$. Hence, by Lemma \ref{L:Foran-lemma}, $A \cap X_M$ is non-$\sigma$-$\P'$-porous in the space $X_M$.
\end{proof}

\begin{remark}
Let $X$ be a complete metric space and  $\P_{up}$ be the porosity-like relation defined by
(for the definition of the upper porosity, see for example \cite{Cuth})
\[
{\P_{up}} = \{(x,A) \in X \times \mathcal{P}(X) \setsep A \text{ is upper porous at the point } x\}.
\]
Let us fix a $(*)$-elementary submodel $M$ containing $X$ and $\P_{up}$. Denote by ${\P'_{up}}$ the porosity-like relation defined by
\[
{\P'_{up}} = \{(x,A) \in X_M \times \mathcal{P}(X_M) \setsep A \text{ is upper porous at the point } x \text{ in the space } X_M\}.
\]
Then, by results from \cite{Cuth} and \cite{Cuth-Rmoutil}, $M$ is a pointwise $({\P_{up}} \leftrightarrow {\P'_{up}})$-model.
It is easy to see that $N_{\P_{up}}(S)$ is a Suslin set whenever $S\subset X$ is Suslin. Thus, by Proposition \ref{P:reduction-porosity-like},
$\sigma$-upper porosity is separably determined property. This result has already been proved in \cite{Cuth-Rmoutil}.
However, a nontrivial inscribing theorem (\cite[Theorem 3.1]{Zeleny-Pelant}) is needed in the proof there. The above mentioned method enables us
to avoid the usage of this result.
\end{remark}

\begin{remark}
It is known to the authors that also the notions of lower porosity, $\langle g\rangle$-porosity, and $(g)$-porosity
satisfy the assumptions of Proposition \ref{P:reduction-porosity-like} (for definitions see \cite{Zajicek-1}).
Consequently, those porosities (and corresponding $\sigma$-porosities) are separably determined when taking Suslin sets in complete metric spaces.
We do not present proofs of those results here since now we see no interesting applications.

Note that in the following section we prove that also the notion of $\alpha$-cone porosity in Asplund spaces satisfies the assumptions of
Proposition~\ref{P:reduction-porosity-like} and, therefore, cone smallness is separably determined.
\end{remark}

\begin{question}
It is an open problem whether the notion of $\sigma$-directional porosity (see \cite{Zajicek-4} for the definition) is separably determined
in the sense of Corollary \ref{C:cone-small}.
\end{question}

Note that the notion of $\sigma$-directional porosity is defined also in \cite{LPT}, but in a slightly different
way which is equivalent to the definition from \cite{Zajicek-4} only in separable Banach spaces.

\section{Cone porosity}\label{S:cone}

In the following section we prove that the notion of $\alpha$-cone porosity in Asplund spaces satisfies the assumptions of Proposition~\ref{P:reduction-porosity-like} and, therefore, cone smallness is separably determined. First, let us give the definition.

\begin{definition}
Let $X$ be a Banach space. For $x^*\in X^*\setminus\{0\}$ and $\alpha\in[0,1)$ we define the \emph{$\alpha$-cone}
\[
C(x^*,\alpha) = \{x\in X \setsep \alpha\|x\|\cdot\|x^*\| < x^*(x)\}.
\]
A set $A \subset X$ is said to be \emph{$\alpha$-cone porous} at $x \in X$ in the space $X$ if there exists $R>0$ such that
for each $\varepsilon > 0$ there exists $z \in U(x,\varepsilon)$ and $x^* \in X^* \setminus \{0\}$ such that
\[
U(x,R) \cap \bigl(z+C(x^*,\alpha)\bigr) \cap A = \emptyset.
\]
The corresponding porosity-like relation is denoted by ${\P_X^{\alpha\text{-}cone}}$.
A set is said to be \emph{cone small} if it is $\sigma$-${\P_X^{\alpha\text{-}cone}}$-porous for each $\alpha\in(0,1)$.
\end{definition}

The following lemma comes from {\cite[Lemma 4.14]{Cuth}}.

\begin{lemma}\label{L:sup-fin-M}
For any suitable elementary submodel $M$ the following holds.
Let $(X,\rho)$ be a metric space and $f \colon X \to \er$ be a function. Then whenever $M$ contains $X$ and $f$,
it is true that for every $R > 0$ and $x \in X_M$ we have
\[
\sup_{u\in U(x,R)}f(u) = \sup_{u\in U(x,R) \cap X_M} f(u).
\]
\end{lemma}

\begin{proposition}\label{P:pointwise-cone-first}
For any suitable elementary submodel $M$ the following holds.
Let $X$ be a Banach space and $\alpha \in [0,1)$. Then whenever $M$ contains $X$ and $\alpha$, $M$ is a pointwise $({\P_X^{\alpha\text{-}cone}}\to {\P_{X_M}^{\alpha\text{-}cone}})$-model.
\end{proposition}

\begin{proof}
Let us fix a $(*)$-elementary submodel $M$ containing $X$, $\alpha$ and a set $A \in \mathcal P(X) \cap M$.
Fix some $x\in X_M$ such that $A$ is $\alpha$-cone porous at $x$. Then there exists a rational number $R > 0$ such that
\[
\forall \varepsilon > 0\; \exists z\in U(x,\varepsilon)\; \exists x^*\in X^*\setminus\{0\}\colon  U(x,R) \cap \bigl(z+C(x^*,\alpha)\bigr) \cap A = \emptyset.
\]
We will show that this formula is true in the space $X_M$ with the constant $\frac14 R$ instead of $R$.
Let us fix a rational $\varepsilon > 0$. Fix a number $0 < \delta < \min\{\frac13 \varepsilon,\frac14 R\}$ and a point $x' \in U(x,\delta) \cap M$.
Then it is easy to observe that the following formula is true
\begin{equation*}\tag{$*$}
\exists z'\in U(x',\tfrac23 \varepsilon)\;\exists x^*\in X^*\setminus\{0\} \colon  U(x',\tfrac12 R) \cap \bigl(z'+C(x^*,\alpha)\bigr) \cap A =  \emptyset.
\end{equation*}
(Indeed, it is enough to take a point $z \in U(x,\tfrac13 \varepsilon) \subset U(x',\tfrac23 \varepsilon)$ and $x^* \in X^* \setminus \{0\}$
satisfying $U(x,R) \cap \bigl(z+C(x^*,\alpha)\bigr) \cap A = \emptyset$ and to observe that $U(x',\tfrac12 R) \subset U(x,R)$.)
Using the absoluteness of this formula (and its subformulas) we find $z' \in U(x',\frac23 \varepsilon) \cap M \subset U(x,\varepsilon) \cap M$ and $x^*\in X^* \cap M \setminus \{0\}$ such that
\begin{equation}\label{E:cone-empty}
U(x',\tfrac{R}{2}) \cap \bigl(z'+C(x^*,\alpha)\bigr) \cap A = \emptyset.
\end{equation}
By Lemma \ref{L:sup-fin-M} we have $\|x^*\| = \|{x^*\rest_{X_M}}\|$.
Hence, the cone $C({x^*\rest_{X_M}},\alpha)$ in the space $X_M$ equals to $C(x^*,\alpha) \cap X_M$. We need to verify that
\[
U(x,\tfrac{R}{4}) \cap \bigl(z'+C(x^*,\alpha)\bigr) \cap A \cap X_M = \emptyset.
\]
Fix some $a \in A \cap X_M$ such that $\|x-a\| < \tfrac14 R$.
Then $a$ is an element of $U(x',\frac12 R)$. By \eqref{E:cone-empty} we conclude $a \notin z'+C(x^*,\alpha)$ and the proof is finished.
\end{proof}

In order to show existence of a pointwise $({\P_X^{\alpha\text{-}cone}} \leftarrow {\P_{X_M}^{\alpha\text{-}cone}})$-models we restrict our attention to Asplund spaces. First, we need to prove that ``functionals from the model $M$ are dense in $(X_M)^*$ when $X$ is an Asplund space''.
This seems to be a nontrivial and very useful result and it might be useful in other separable reduction theorems as well.
The proof can be done using the existence of a ``projectional generator with domain $X$'' in the dual space of an Asplund space $X$.
In fact, it is sufficient to use only the first part of the proof of this statement from \cite{Fabian}.

\begin{theorem}\label{T:Asplund-model}
For any suitable elementary submodel $M$ the following holds.
Let $X$ be an Asplund space. Then whenever $M$ contains $X$, it is true that
$\ov{\{x^*\rest_{X_M} \setsep x^*\in X^*\cap M\}} = (X_M)^*$.
\end{theorem}

\begin{proof}
It is obvious that the inclusion ``$\subset$'' holds. Let us show that the opposite one holds as well.
It is proved in the second step of the proof of {\cite[Theorem 1]{Fabian}}
that there exists continuous mappings $D(n) \colon X \to X^*$, $n \in \en$, such that
that, for every closed separable subspace $V \subset X$, we have
\[
\ov{\sspan}\{D(n)(x) \rest_V \setsep n \in \en, x \in V\} = V^*.
\]
Using the absoluteness of the formula (and its subformula)
\begin{equation*}\tag{$*$}
\begin{split}
\exists D \colon \bigl(& D \text{ is a function, } \dom D = \en, \ D(n) \text{ are $\|\cdot\|$-$\|\cdot\|$ continuous mappings} \\
&\text{from $X$ into $X^*$ and for every closed separable subspace $V$ of $X$ we have } \\
&\ov{\sspan}\{D(n)(x) \rest_V \setsep n \in \en, x \in V\} = V^* \bigr),\\
\end{split}
\end{equation*}
we may without loss of generality assume that $D(n) \in M$ for every $n \in \en$. Thus, for every $n \in \en$ and $x \in X \cap M$, we have $D(n)(x) \in M$.
Using the continuity of $D(n)$ for every $n \in \en$, we get
\[
\{D(n)(x) \setsep n \in \en, x \in X_M\} \subset \ov{\{D(n)(x) \setsep n \in \en, x \in X \cap M\}} \subset \ov{X^* \cap M}.
\]
Finally,
\begin{align*}
(X_M)^* &= \ov{\sspan}\{{D(n)(x)\rest_{X_M}} \setsep n \in \en, x \in X_M\}
\subset \ov{\{x^*\rest_{X_M} \setsep  x^* \in \ov{X^* \cap M}\}} \\
&\subset \ov{\{x^*\rest_{X_M} \setsep x^* \in X^* \cap M\}}.
\end{align*}
\end{proof}

Now, we need to observe that it is enough to consider functionals from a dense subset of $X^*$ in the definition of $\alpha$-cone porosity.

\begin{lemma}\label{L:equivalent-def-cone}
Let $X$ be a Banach space and let $E\subset X$ and $D\subset X^*$ be norm-dense subsets. Let $A\subset X$ and $x\in X$ and $\alpha\in [0,1)$. Then $A$ is $\alpha$-cone porous at $x$ if and only if the following is true:
\[
\exists R\in \qe_+ \ \forall \varepsilon\in \qe_+ \ \exists y^* \in D \ \exists w \in B(x,\varepsilon) \cap E \colon B(x,R) \cap (w + C(y^*,\alpha)) \cap A = \emptyset.
\]
\end{lemma}

\begin{proof}
The sufficiency of our condition is easy to see. Let us, therefore, assume that a given set $A$ is $\alpha$-cone porous at a given point $x\in X$, and deduce from it the desired condition.

Since $A$ is $\alpha$-cone porous at $x$, it is easy to see that there exists $R\in \qe_+$ such that 
\begin{equation}\label{E:cone-porous}
\forall \varepsilon > 0 \ \exists x^* \in X \ \exists z \in B(x,\varepsilon) \colon B(x,R) \cap \bigl(z + C(x^*,\alpha)\bigr) \cap A = \emptyset.
\end{equation}
Let $\varepsilon\in \qe_+$. Using \eqref{E:cone-porous} we find $x^* \in X^*$ and $z \in B(x,\min\{\varepsilon,R\})$ such that 
\[
B(x,R) \cap (z + C(x^*,\alpha)) \cap A = \emptyset.
\]
Choose $w \in B(x,\min\{\varepsilon,R\}) \cap (z + C(x^*,\alpha))\cap E$. Since $w-z \in C(x^*,\alpha)$, we have 
\[
x^*(w-z) - \alpha \|x^*\| \|w-z\| > 0.
\]
Using the last inequality and density of $D$ we find $y^* \in D$ such that 
\begin{itemize}
\item[(a)] $\|y^*\| \geq \|x^*\|$,

\item[(b)] $\|x^*-y^*\| < \frac{1}{2R}\bigl(x^*(w-z)-\alpha\|x^*\|\|w-z\|\bigr)$.
\end{itemize}
Now it is sufficient to prove that 
\begin{equation}\label{E:inclusion-of-cones}
B(x,R) \cap (w + C(y^*,\alpha)) \subset z + C(x^*,\alpha).
\end{equation}
Indeed, since then we have 
\[
B(x,R) \cap (w + C(y^*,\alpha)) \cap A \subset B(x,R) \cap (z + C(x^*,\alpha)) \cap A = \emptyset.
\]
To verify \eqref{E:inclusion-of-cones} take $u \in C(y^*,\alpha)$ with $w + u \in B(x,R)$.
Then we have
\begin{equation}\label{E:estimate-2R}
\|u\| = \|(u+w-x)+(x-w)\| \leq \|u+w-x\| + \|x-w\| \leq 2R.
\end{equation}
We compute
\[
\begin{split}
x^*(w+u-z) &= x^*(w-z) + x^*(u) \\
&= x^*(w-z) + y^*(u) + (x^*-y^*)(u) \\
&> x^*(w-z) + \alpha\|y^*\| \cdot \|u\| - \|x^*-y^*\|\cdot \|u\|  \\
&\geq x^*(w-z) + \alpha \|x^*\| \cdot \|u\| - \|x^*-y^*\|\cdot 2R \qquad (\text{by} \eqref{E:estimate-2R}) \\
&\geq \alpha \|x^*\| \cdot \|u\| + \alpha \|x^*\| \cdot \|w-z\|   \qquad (\text{by (b)}) \\
&\geq \alpha \|x^*\| \cdot \|u+w-z\|.
\end{split}
\]
This shows that $w+u-z \in C(x^*,\alpha)$. Consequently, we get $w + u \in z + C(x^*,\alpha)$.
\end{proof}

Now, we are ready to see the existence of a pointwise $({\P_X^{\alpha\text{-}cone}} \leftrightarrow {\P_{X_M}^{\alpha\text{-}cone}})$-models in Asplund spaces.

\begin{proposition}\label{P:point-cone}
For any suitable elementary submodel $M$ the following holds.
Let $X$ be an Asplund space and $\alpha \in [0,1)$. Then whenever $M$ contains $X$ and $\alpha$, 
$M$ is a pointwise $({\P_X^{\alpha\text{-}cone}}\leftrightarrow {\P_{X_M}^{\alpha\text{-}cone}})$-model.
\end{proposition}

\begin{proof}
Let us fix a $(*)$-elementary submodel $M$ containing $X$, $\alpha$, and a set $A \subset X$.
By Proposition~\ref{P:pointwise-cone-first}, $M$ is a pointwise $({\P_X^{\alpha\text{-}cone}}\to {\P_{X_M}^{\alpha\text{-}cone}})$-model.
Fix some $x\in X_M$ such that $A$ is not $\alpha$-cone porous at $x$.
We will show that $A\cap X_M$ is not $\alpha$-cone porous at $x$ in the space $X_M$.
Notice that, by Lemma~\ref{L:sup-fin-M}, $\|{x^*\rest_{X_M}}\| = \|x^*\|$ for every $x^*\in X^* \cap M$.
Hence, the cone $C({x^*\rest_{X_M}},\alpha)$ in the space $X_M$ equals $C(x^*,\alpha)\cap X_M$.
Thus, by Lemma~\ref{L:equivalent-def-cone} and Theorem~\ref{T:Asplund-model}, it is sufficient to prove that the following formula is true
\begin{align*}
\forall R\in\qe_+\;\exists\varepsilon\in\qe_+\;\forall z\in U(x,\varepsilon)\cap M\;
& \forall x^* \in (X^*\cap M) \setminus\{0\}\colon \\
&A \cap U(x,R) \cap X_M \cap \bigl(z+C(x^*,\alpha)\bigr) \neq \emptyset.
\end{align*}
Fix $R \in \qe_+$. As $A$ is not $\alpha$-cone porous at $x$, there exists $\varepsilon \in \qe_+$ such that
\begin{equation}\label{E:cone-in-X}
\forall z \in U(x,\varepsilon)\; \forall x^* \in X^* \setminus \{0\} \colon A \cap U(x,\tfrac13 R) \cap \bigl(z+C(x^*,\alpha)\bigr) \neq \emptyset.
\end{equation}
Let us fix $z\in U(x,\varepsilon)\cap M$ and $x^*\in (X^* \cap M) \setminus \{0\}$. Find some $x' \in U(x,\tfrac13 R) \cap M$.
Then $U(x,\tfrac13 R)\subset U(x',\tfrac23 R)$. By \eqref{E:cone-in-X}, the following formula is true
\begin{equation*}\tag{$*$}
\exists a \in A \colon a \in \bigl(z+C(x^*,\alpha)\bigr) \cap U(x',\tfrac23 R).
\end{equation*}
Using the absoluteness of the formula (and its subformula) above, there exists $a \in A \cap M$ satisfying the formula above.
It is easy to verify that $a\in U(x,R) \cap \bigl(z+C(x^*,\alpha)\bigr)$. Hence,
\[
A \cap U(x,R) \cap X_M \cap \bigr(z+C(x^*,\alpha)\bigl) \neq \emptyset.
\]
Thus, $A \cap X_M$ is not $\alpha$-cone porous at $x$ in the space $X_M$. This finishes the proof.
\end{proof}

In the remainder of the section we prove that the assumption on descriptive quality of $N_{\P}(S)$ from Proposition \ref{P:reduction-porosity-like} is satisfied for cone porosity.
We begin with the following lemma.

\begin{lemma}\label{L:dist-cone}
Let $X$ be a Banach space, $x^* \in X^*$, and $\alpha \in [0,1)$. Then for each $x\in C(x^*,\alpha)$ it is true that
$\d(X\setminus C(x^*,\alpha), x+C(x^*,\alpha)) > 0$.
\end{lemma}

\begin{proof}
It is easy to verify that $C(x^*,\alpha)$ is an open set and that it is a convex cone in the sense
that for any two points $y,z$ from $C(x^*,\alpha)$ and any $c > 0$ the points $cy$ and $y+z$ also belong to $C(x^*,\alpha)$.
Let $x \in C(x^*,\alpha)$. Set $\delta := \d(x,X \setminus C(x^*,\alpha))$. The number $\delta$ is positive,
since  $C(x^*,\alpha)$ is open. Take any point $y \in x+C(x^*,\alpha)$ (then $y-x\in C(x^*,\alpha)$).
Hence, $U(x,\delta)\subset C(x^*,\alpha)$, and so $U(y,\delta) = (y-x)+U(x,\delta)\subset C(x^*,\alpha)$.
Since $y\in x+C(x^*,\alpha)$ was chosen arbitrarily, we conclude that $\d(X\setminus C(x^*,\alpha), x+C(x^*,\alpha))\geq\delta>0$.
\end{proof}

\begin{proposition}\label{L:points-cone-porous}
Let $X$ be a Banach space, $\alpha \in [0,1)$, and $A \subset X$ be any set.
Then the set $S$ of all points $x \in X$ at which $A$ is $\alpha$-cone porous is Borel (of the type $G_{\delta\sigma}$).
\end{proposition}

\begin{proof}
For $x,z \in X$, $R > 0$, $x^* \in X^* \setminus \{0\}$, and $\alpha \in [0,1)$ we set
\[
T(x,R,z,x^*,\alpha) = U(x,R) \cap \bigl(z+C(x^*,\alpha)\bigr).
\]
First we show that
\begin{equation}\label{E:cone-porosity}
S = \bigcup_{R\in\qe_+} \bigcap_{\varepsilon\in\qe_+} \bigcup_{x^*\in X^*\setminus \{0\}} G(R,\varepsilon,x^*),
\end{equation}
where
\[
G(R,\varepsilon,x^*) = \{x \in X \setsep \exists z \in U(x,\varepsilon)\colon \d(T(x,R,z,x^*,\alpha),A) > 0 \}.
\]
It is easy to see that the inclusion ``$\supset$'' holds. To prove the opposite one consider $x \in S$.
Then we can find $R' > 0$ such that for every $\varepsilon > 0$ there is a $z' \in U(x,\varepsilon)$ and $x^* \in X^*\setminus \{0\}$ such that $T(x,R,z',x^*,\alpha) \cap A = \emptyset$. Fix $R \in (0,R') \cap \qe$ and take any $\varepsilon \in \qe_+$.
Then we find $z' \in U(x,\varepsilon)$ and $x^* \in X^* \setminus \{0\}$ such that $T(x,R,z',x^*,\alpha) \cap A = \emptyset$.
For $z \in \bigl(z' + C(x^*,\alpha)\bigr) \cap U(x,\varepsilon)$ we have $\d(T(x,R,z,x^*,\alpha),X \setminus U(x,R')) > 0$
and by Lemma \ref{L:dist-cone} we have that
\[
\d\bigl(T(x,R,z,x^*,\alpha),X \setminus (z'+C(x^*,\alpha))\bigr) > 0.
\]
Since
\[
A \subset \bigl(X \setminus U(x,R')\bigr) \cup \bigl(X \setminus (z' + C(x^*,\alpha)\bigr),
\]
we get $x \in G(R,\varepsilon,x^*)$ and the equality \eqref{E:cone-porosity} is proved.

Now it is sufficient to prove that the set $G(R,\varepsilon,x^*)$ is open.
To this end fix $R > 0$, $\varepsilon > 0$, $x^* \in X^* \setminus \{0\}$ and consider $x \in G(R,\varepsilon,x^*)$.
There exists $z \in U(x,\varepsilon)$ with $\d(T(x,R,z,x^*,\alpha),A) > 0$.
Denote $\eta = \d(T(x,R,z,x^*,\alpha),A)$. We have $\eta > 0$. For $x' \in U(x,\frac12 \eta)$ we have
\[
T(x',R,z+x'-x,x^*,\alpha) = (x'-x) + T(x,R,z,x^*,\alpha).
\]
This gives $\d(T(x',R,z+x'-x,x^*,\alpha),A) \geq \frac12 \eta > 0$.
Since $z+x'-x \in U(x',\varepsilon)$ we have $x' \in G(R,\varepsilon,x^*)$.
This implies $U(x,\frac12 \eta) \subset  G(R,\varepsilon,x^*)$ and we are done.
\end{proof}

\begin{cor}\label{C:cone-por-suslin}
Let $X$ be a Banach space, $\alpha\in [0,1)$, and $A \subset X$ be a Suslin set. Then the set $N_{\P_X^{\alpha\text{-}cone}}(A)$ is Suslin.
\end{cor}

\begin{theorem}\label{T:cone}
For any suitable elementary submodel $M$ the following holds.
Let $X$ be an Asplund space, $A \subset X$ be Suslin, and $\alpha\in [0,1)$.
Then whenever $M$ contains $X$, $A$, and $\alpha$, the following are true:
\begin{align*}
A \text{ is $\alpha$-cone porous in the space } X           &\iff A \cap X_M \text{ is $\alpha$-cone porous in the space } X_M, \\
A \text{ is $\sigma$-$\alpha$-cone porous in the space } X  &\iff A \cap X_M \text{ is $\sigma$-$\alpha$-cone porous in the space } X_M, \\
A \text{ is cone small in the space } X                     &\iff A \cap X_M \text{ is cone small in the space } X_M.
\end{align*}
\end{theorem}

\begin{proof}
Let us fix a $(*)$-elementary submodel $M$ containing $X$, $A$, and $\alpha$. Then the following formula is clearly true
\begin{equation*}\tag{$*$}
\exists {\bf R} \text{ point-set relation on } X \; \forall x \in X \; \forall B \subset X \colon
(B \text{ is $\alpha$-cone porous at } x \iff (x,B)\in {\bf R}).
\end{equation*}
The absoluteness of this formula and its subformula implies that ${\P_X^{\alpha\text{-}cone}} \in M$.
The first two parts of the theorem now follow using Propositions \ref{P:reduction-porosity-like} and \ref{P:point-cone} and Corollary \ref{C:cone-por-suslin}.
To prove the last part one just needs to observe that any set in any Banach space is cone small if and only if it is $\sigma$-$\beta$-cone porous for each $\beta \in (0,1) \cap \qe$.
\end{proof}

\begin{cor}\label{C:cone-small}
Let $X$ be an Asplund space and $A \subset X$ be a Suslin set.
Then for every separable space $V_0\subset X$ there exists a closed separable space $V \subset X$ such that
$V_0 \subset V$ and
\[
A \text{ is cone small in } X \iff A \cap V \text{ is cone small in } V.
\]
\end{cor}

\section{Applications}

\begin{definition}[\cite{NLT}]
Let $X$ be a real Banach space, $G \subset X$ be open. A function $f\colon G \to \er$ is called \emph{approximately convex at $x_0 \in G$}
if for every $\varepsilon >0$ there exists $\delta>0$ such that
\[
f(\lambda x + (1-\lambda)y)\leq \lambda f(x) + (1-\lambda) f(y) + \varepsilon \lambda(1-\lambda) \| x-y \|
\]
whenever $\lambda \in [0,1]$ and $x,y \in U(x_0,\delta)$.
We say $f$ is \emph{approximately convex on $G$} if it is approximately convex at each $x_0 \in G$.
\end{definition}

\begin{remark}
The class of approximately convex functions includes semiconvex functions and strongly paraconvex functions
(for definitions see, e.g., \cite{Zajicek-5}).
\end{remark}

We shall apply our result about cone small sets to prove the following generalization of \cite[Theorem 5.5]{Zajicek-5} to nonseparable Asplund spaces.
Note that the following theorem is also a strengthening of \cite[Theorem 5.9]{Zajicek-5} which states that a continuous approximately convex function
on an Asplund space is Fr\'echet differentiable except for points from a union of a cone small set and a $\sigma$-cone supported set.
Note also that unlike \cite[Theorem 5.5]{Zajicek-5}, our Theorem \ref{T:appl-approx-convexity} states that the exceptional set is cone small and not angle small. However, these two notions are equivalent if $X$ is separable.

\begin{theorem}\label{T:appl-approx-convexity}
Let $X$ be an Asplund space and $G \subset X$ be open. Let $f\colon G \to \er$ be a continuous and approximately convex function.
Then the set $N_F(f)$ of all points of $G$ at which $f$ is not Fr\'echet differentiable is cone small.
\end{theorem}

To prove the theorem we will need several notions and a lemma.
The notion of LAN mapping is defined and studied in \cite{Zajicek-2}.

\begin{definition}
Let $X$ be a Banach space and $G \subset X$ be open.
We say a (singlevalued) mapping $g\colon G \to X^*$ is \emph{LAN} (\emph{locally almost nonincreasing}) if for any $a \in G$ and $\varepsilon > 0$
there exists $\delta > 0$ such that for any $x_1, x_2 \in U(a,\delta)$ we have
\[
\langle g(x_1)-g(x_2), x_1-x_2 \rangle \leq \varepsilon \| x_1 - x_2 \|.
\]

We say a multivalued mapping $T \colon G \to X^*$ is \emph{submonotone on $G$}
if for any $a\in G$ and $\varepsilon > 0$ there exists $\delta > 0$ such that for any $x_1,x_2 \in U(a,\delta)$,
$x_1^* \in T(x_1)$, and $x_2^* \in T(x_2)$ we have
\[
\langle x_1^*-x_2^*, x_1-x_2 \rangle \geq - \varepsilon \| x_1 - x_2 \|.
\]
\end{definition}

\begin{remark}\label{R:LAN-monot}
Clearly $T$ is LAN if and only if $-T$ is singlevalued and submonotone.
\end{remark}

The following lemma generalizes \cite[Lemma 3]{Zajicek-2} to general Asplund spaces.

\begin{lemma}\label{L:red-monot}
Let $X$ be Asplund, $G \subset X$ be open and $g\colon G \to X^*$ be LAN.
Then $g$ is continuous at all points of $G$ except those which belong to a cone small set.
\end{lemma}

\begin{proof}
Denote by $A$ the set of all points of $G$ at which $g$ is not continuous (then $A$ is Borel) and
let us fix a $(*)$-elementary submodel $M$ containing $X$ and $g$. Then $X_M$ is a Banach space with separable dual and $g\rest_{X_M}$ is clearly LAN. Denote by $B$ the set of all points of $G\cap X_M$ (the intersection is nonempty) at which $g\rest_{X_M}$ is not continuous.
By \cite[Lemma 3]{Zajicek-2}, $B$ is angle small in $X_M$.  But \cite[Theorem 5.1]{Cuth} gives that $B = A \cap X_M$ and that $A\in M$.
Hence, by Theorem \ref{T:cone}, $A$ is cone small.
\end{proof}

\begin{definition}
Let $X$ be a Banach space, $G\subset X$ be open, and $f \colon G \to \er$. The Fr\'echet subdifferential of $f$ at $a$ is defined by
\[
\partial^F f(a) = \Bigl\{ x^*\in X^*\setsep \liminf_{h\to 0} \frac{f(a+h)-f(a)-\langle x^*, h \rangle}{\| h \|}\geq 0\Bigr\}.
\]
\end{definition}

\begin{proof}[Proof of Theorem \ref{T:appl-approx-convexity}]
By \cite[Lemma 2.5 (ii) and (iii)]{Zajicek-5} the multivalued mapping $x\mapsto \partial^F f(x)$ is submonotone on $G$. Choose any selection $g$ of $\partial^F f$ on $G$; then $g$ is also submonotone. Lemma~\ref{L:red-monot} (together with Remark~\ref{R:LAN-monot}) implies that $g$ is continuous on $G$ up to a cone small set. Now, \cite[Lemma 5.4]{Zajicek-5} says that $f$ is Fr\'echet differentiable at points of continuity of $g$, concluding the proof.
\end{proof}

Another possible application of Theorem \ref{T:cone} is the following strengthening of \cite[Proposition 4.2]{Vesely}
(for definitions see \cite{Vesely}).

\begin{proposition}
Let $Y$ be a countably Daniell ordered Banach space with the Radon-Nikod\'ym property. Assume that
\begin{itemize}
\item[(a)] either $X$ is a closed subspace of $c_0(\Delta)$, where $\Delta$ is an uncountable set,

\item[(b)] or $X=C(K)$, where $K$ is scattered compact topological space.
\end{itemize}
Let $A \subset X$ be an open convex set and $f \colon A \to Y$ be a continuous convex operator. Then $f$ is Fr\'echet differentiable on $A$ except for a cone small $\Gamma$-null set.
\end{proposition}

The only difference from the original assertion is that, instead of $\sigma$-lower porous, we have the exceptional set cone small which is a stronger assertion.
We shall, however, omit the proof, as there is no difference from the proof in \cite{Vesely};
one just needs to use our Theorem \ref{T:cone} instead of \cite[Theorem~5.4]{Cuth-Rmoutil} which is an analogue of \ref{T:cone} for $\sigma$-lower porosity.

Note that we also obtain an analogue of \cite[Proposition CR]{Vesely} for cone smallness. Since this could be of some independent interest,
it is, perhaps, worth stating (see also \cite[Theorem 1.2]{Cuth-Rmoutil}).

\begin{proposition}
Let $X, Y$ be Banach spaces, $G \subset X$ be an open set, and $f \colon G \to Y$ an arbitrary mapping.
Then for every separable space $V_0 \subset X$ there exists a closed separable space $V \subset X$ such that $V_0 \subset V$
and that the following are equivalent:
\begin{itemize}
\item[(i)]  the set of all points where $f$ is not Fr\'echet differentiable is cone small in $X$,

\item[(ii)] the set of all points where $f\rest _{V \cap G}$ is not Fr\'echet differentiable is cone small in $V$.
\end{itemize}
\end{proposition}


\begin{thebibliography}{10}

\bibitem{Cuth}
Marek C{\'u}th.
\newblock Separable reduction theorems by the method of elementary submodels.
\newblock {\em Fund. Math.}, 219(3):191--222, 2012.

\bibitem{Cuth-Rmoutil}
Marek C{\'u}th and Martin Rmoutil.
\newblock {$\sigma$}-porosity is separably determined.
\newblock {\em Czechoslovak Math. J.}, 63(138)(1):219--234, 2013.

\bibitem{Fabian}
Mari{\'a}n Fabi{\'a}n and Gilles Godefroy.
\newblock The dual of every {A}splund space admits a projectional resolution of
  the identity.
\newblock {\em Studia Math.}, 91(2):141--151, 1988.

\bibitem{Kunen}
Kenneth Kunen.
\newblock {\em Set theory}, volume 102 of {\em Studies in Logic and the
  Foundations of Mathematics}.
\newblock North-Holland Publishing Co., Amsterdam, 1980.
\newblock An introduction to independence proofs.

\bibitem{LPT}
Joram Lindenstrauss, David Preiss, and Jaroslav Ti{\v{s}}er.
\newblock {\em Fr\'echet differentiability of {L}ipschitz functions and porous
  sets in {B}anach spaces}, volume 179 of {\em Annals of Mathematics Studies}.
\newblock Princeton University Press, Princeton, NJ, 2012.

\bibitem{NLT}
Huynh~Van Ngai, Dinh~The Luc, and Michel Th{\'e}ra.
\newblock Approximate convex functions.
\newblock {\em J. Nonlinear Convex Anal.}, 1(2):155--176, 2000.

\bibitem{Rmoutil}
Martin Rmoutil.
\newblock Products of non-{$\sigma$}-lower porous sets.
\newblock {\em Czechoslovak Math. J.}, 63(138)(1):205--217, 2013.

\bibitem{Vesely}
Libor Vesel{\'y} and Lud{\v{e}}k Zaj{\'{\i}}{\v{c}}ek.
\newblock On differentiability of convex operators.
\newblock {\em J. Math. Anal. Appl.}, 402(1):12--22, 2013.

\bibitem{Zajicek-1}
Lud{\v{e}}k Zaj{\'{\i}}{\v{c}}ek.
\newblock Sets of {$\sigma $}-porosity and sets of {$\sigma $}-porosity
  {$(q)$}.
\newblock {\em \v Casopis P\v est. Mat.}, 101(4):350--359, 1976.

\bibitem{Zajicek-2}
Lud{\v{e}}k Zaj{\'{\i}}{\v{c}}ek.
\newblock On the {F}r\'echet differentiability of distance functions.
\newblock In {\em Proceedings of the 12th winter school on abstract analysis
  ({S}rn\'\i , 1984)}, number Suppl. 5, pages 161--165, 1984.

\bibitem{Zajicek-3}
Lud{\v{e}}k Zaj{\'{\i}}{\v{c}}ek.
\newblock Porosity and {$\sigma$}-porosity.
\newblock {\em Real Anal. Exchange}, 13(2):314--350, 1987/88.

\bibitem{Zajicek-4}
Lud{\v{e}}k Zaj{\'{\i}}{\v{c}}ek.
\newblock On {$\sigma$}-porous sets in abstract spaces.
\newblock {\em Abstr. Appl. Anal.}, (5):509--534, 2005.

\bibitem{Zajicek-5}
Lud{\v{e}}k Zaj{\'{\i}}{\v{c}}ek.
\newblock Differentiability of approximately convex, semiconcave and strongly
  paraconvex functions.
\newblock {\em J.~Convex~Anal.}, 15(1):1--15, 2008.

\bibitem{Zeleny-Pelant}
Miroslav Zelen{\'y} and Jan Pelant.
\newblock The structure of the {$\sigma$}-ideal of {$\sigma$}-porous sets.
\newblock {\em Comment. Math. Univ. Carolin.}, 45(1):37--72, 2004.

\end{thebibliography}
\end{document}